 \documentclass[preprint,12pt,nonatbib]{elsarticle}

\usepackage{lineno}

\usepackage{amsmath, amsthm,amsfonts,amssymb}

\usepackage{url}
\usepackage{xcolor}

\newtheorem{lemma}{Lemma}
\newtheorem{theorem}[lemma]{Theorem}
\newtheorem{definition}[lemma]{Definition}
\newtheorem{remark}[lemma]{Remark}
\newtheorem{corollary}[lemma]{Corollary}

\journal{}

 \usepackage{numcompress}

\begin{document}

\begin{frontmatter}
\title{
Longterm existence of solutions of a reaction diffusion system with non-local terms modeling an immune response - an interpretation-orientated proof }
\author[1]{Cordula Reisch\corref{cor1}}
\ead{c.reisch@tu-bs.de}
\author[1] {Dirk Langemann}
\ead{d.langemann@tu-bs.de}
\cortext[cor1]{Corresponding author}

\address[1]{TU Braunschweig, Institute for Partial Differential Equations, Universit\"atsplatz 2, 38106 Braunschweig, Germany}

\begin{abstract}
This paper shows the global existence and boundedness of solutions of a reaction diffusion system modeling liver infections. 
The existence proof is presented step by step and the focus lies on the interpretation of intermediate results in the context of liver infections which is modeled. 
Non-local effects in the dynamics between the virus and the immune system cells coming from the immune response in the lymphs lead to an integro-partial differential equation. 
While existence theorems for parabolic partial differential equations are textbook examples in the field, the additional integral term requires new approaches to proving the global existence of a solution.
This allows to set up an existence proof with a focus on interpretation leading to more insight in the system and in the modeling perspective at the same time. 

We show the boundedness of the solution in the $L^1(\Omega)$- and the $L^2(\Omega)$-norms, and use these results to prove the global existence and boundedness of the solution. 
A core element of the proof is the handling of oppositely acting mechanisms in the reaction term, which occur in all population dynamics models and which results in reaction terms with opposite monotonicity behavior. 
In the context of modeling liver infections, the boundedness in the $L^\infty(\Omega)$-norm has practical relevance:
Large immune responses lead to strong inflammations of the liver tissue.
Strong inflammations negatively impact the health of an infected person and lead to grave secondary diseases.
The gained rough estimates are compared  with numerical tests. 
\end{abstract}


\end{frontmatter}

\section{Introduction}
Modeling the coupled dynamics of virus and the immune system during a liver infection caused by a hepatitis virus is challenging because the mechanisms behind persisting infections over month or years are still unknown, \cite{thomas_experimental_2016}.
An opportunity for overcoming the problem of unknown mechanisms on the cell scale contains two integrative changes. 
First, the modeling scale is changed from the cell scale towards a mesoscopic scale on the length scale of a few centimeters. 
Second, the mechanisms, which are unknown in detail, are replaced by integrative mechanisms representing the commonly accepted properties of the unknown mechanisms. 
This change of view results in a compact model of partial differential equations. 

Modeling inflammations with differential equations is a widely used approach. 
For example in \cite{ibragimov_mathematical_2005, volpert_elliptic_2011}, atherogenesis as a particular inflammation is modeled with reaction diffusion equations. 
In \cite{ibragimov_mathematical_2005}, instable states are interpreted as persisting infections, whereas in \cite{volpert_elliptic_2011} travelling waves are interpreted as persisting infections. 
A reaction diffusion system for modeling the dynamics of liver infections is presented in \cite{rezounenko_viral_2018}.
In \cite{aston_new_2018,dahari_modeling_2007} systems of ordinary equations are used for modeling the total amount of immune system cells and virus during a hepatitis C liver infection. 

In \cite{kerl_reaction_2012, reisch_chemotactic_2019, reisch_modeling_2019, reisch_entropy_2020, reisch_reaktions-diffusions-gleichungen_2020} liver inflammations are modeled by using reaction diffusion equations describing the virus concentration and the T~cell population during an infection. 
As a specific feature, the reaction diffusion equations include a space-dependent and non-local term describing the inflow of T~cells in a small part of the modeled region. 
The amount of inflowing T~cells depends on the total virus amount in the regarded part of the liver.
The dependency on the total virus amount is represented by an integral term over the whole domain. 
The non-local term models the T~cell dispersal starting in the lymphs. 

The description of the inflow region, called portal field, reflects some important parts of the real liver structure.
Therefore, the term is desirable and necessary for modeling liver infections even if it makes the mathematical analysis of the model more difficult.
 One challenging task caused by the non-local and space-depending inflow term is the proof of the longterm existence of a solution. 
 Often used results for parabolic partial differential equations are based on Lipschitz continuous reaction functions with respect to the state variable or require monotonous reaction functions. 
Due to the integral term and the oppositely acting mechanisms, these results are not directly applicable to the system modeling the dynamics of liver infections, see Sec.~\ref{sec:existence}.

 In this paper, the longterm existence and boundedness of solutions of the model proposed in \cite{kerl_reaction_2012} is proven and the results are interpreted in the light of the application. 
 The focus therefore lies not only on adapting established theorems but on finding interpretable estimations on the way to an existence result. 
 Therefore, the model is presented in Sec.~\ref{sec:reacdiff}.
 An important property of the reaction functions are the oppositely acting mechanisms like in the classical Lotka-Volterra equations and in nearly all population dynamics models.
The non-local term is a new feature compared to the classical model and influences the dynamics of the model much more than only by its position-dependency.  
 
 In Sec.~\ref{sec:existence}, the longterm existence of solutions is proven. 
First, the local existence of a weak solution is concluded from existence results for parabolic differential equations with Dirichlet boundary conditions.
Additionally, properties of the solution like its non-negativity and the boundedness of one state variable are shown. 
Due to the inflow term modeling the arriving of T~cells from the lymphs, showing a-priori boundedness of the second state variable is the main concern. 

The boundedness of the second variable is shown in different steps, starting with proofs of the boundedness of the solution in $L^1(\Omega)$ and $L^2(\Omega)$ in Sec.~\ref{sec:lpbounds}. 
The proofs use different functionals depending on the $L^1(\Omega)$- or $L^2(\Omega)$-norms and they are handling the oppositely acting mechanisms in the reaction function. 
As a result, we get rough but robust estimates for the $L^1(\Omega)$- and $L^2(\Omega)$-norms of the solution.
In the context of liver infections, this result will be interpreted in the light of the total amount of T~cells.

The results are used for proving the boundedness of the solution in $L^\infty(\Omega)$.
Consequently, the global existence of a bounded solution is shown. 
The boundedness of the solution in $L^\infty(\Omega)$ is an important property showing how the mathematical proof evokes insight in the application, which is a liver infection, and vice versa the inflammation application feeds back to the mathematics. 
The immune response in the second state variable, i.e. the amount of T~cells, contains the strength of the inflammation. 
Its upper bound is related to illness and survival of an infected individual. 

In Sec.~\ref{sec:num}, the quality of used estimates is visualized for different solutions types which are interpreted as different infection courses. 
The paper finishes with a conclusion of the results and further ideas.

\section{Reaction diffusion infection model with non-local inflow}\label{sec:reacdiff}

A model for describing the interaction between virus and T~cells during a viral liver infection is presented in \cite{kerl_reaction_2012} and analyzed in \cite{kerl_reaction_2012, reisch_chemotactic_2019, reisch_modeling_2019, reisch_entropy_2020, reisch_reaktions-diffusions-gleichungen_2020}. 
The virus population density $u=u(t,\mathbf{x})$ is named according to the prey in the classical Lotka Volterra model. 
The cells of the immune system are concluded as T~cells.
They can be seen as predator for the virus and are therefore named $v=v(t,\mathbf{x})$. 
We model the interaction in a part of the liver seen as a domain $\Omega \subset \mathbb{R}^d$ with $d= \{ 2, 3 \}$. 

According to \cite{kerl_reaction_2012}, the T~cells, as the summed cells of the immune system, kill infected liver cells and thus the virus. 
Both, the T~cells and the virus spread out in the liver, modeled by diffusion terms. 
The virus grow by reproduction in dependency of the local virus amount. 
The change of the T~cell population depends on the total virus load inside the liver, which is modeled by an inflow term $j[u]$. 

Since the T~cells as immune response are produced in the lymphs outside the liver, the T~cells arrive in the regarded part of the liver through portal fields, which are sub-domains $\Theta\subset\Omega$. 
Furthermore, the external production of the immune response motivates the dependence of the inflow $j=j[u]$ on the total amount of virus in the regarded domain $\Omega$, i.\,e.\ the inflow $j=j[u](\mathbf{x})$ 
in every point $\mathbf{x}\in\Theta$ depends non-locally on the integral $U(t)=\|u(t,\cdot)\|_{L^1(\Omega)}$ of $u$ over $\Omega$.

\begin{remark}[Modeling scale]\label{rem:scale}
In the context of liver infections, the area $\Theta$ can be seen as a model for a portal field through which T~cells enter a certain part of the liver $\Omega$. 
The model abstracts from the cell-scale structure of the liver and the involved cells. 
Nevertheless, we cover some basic structure of a liver by still regarding portal fields in the liver. 
\end{remark}

We regard, as a simplification, the boundary $\partial \Omega$ of the domain $\Omega $ to be impermeable. 
This results in zero flux or homogeneous Neumann boundary conditions.

Using as few mechanisms as possible, see \cite{reisch_entropy_2020}, we find the predator-prey model
\begin{align}
\begin{aligned}\label{eq:sys}
u_{,t} &= u w(u) - \gamma u v + \alpha \Delta u  \ \ && \text{ for } \mathbf{x} \in \Omega, t>0, \\
v_{,t} &= j[u] - \eta  (1 -u) v    + \beta \Delta v \ \ && \text{ for } \mathbf{x} \in \Omega, t>0,\\
 u(0,\mathbf{x} )&= u_0(\mathbf{x}) , \ v(0,\mathbf{x} )=v_0(\mathbf{x}) && \text{ for } \mathbf{x} \in \Omega , \\
 0&= \nabla u \cdot \mathbf{n} =  \nabla v \cdot \mathbf{n}   && \text{ for } \mathbf{x} \in \partial \Omega, t>0
\end{aligned}
\end{align}
with a growth function $w(u)$ describing the non-linear growth of the virus in absence of other mechanisms and the non-local inflow $j=j[u]({\mathbf{x}})$ of T~cells.
The constants $\alpha$ and $\beta$ describe the strength of diffusion. 
The reaction diffusion system in Eq.~(\ref{eq:sys}) contains the predator term $\gamma u v$ diminishing the virus in presence of the immune response $v$, and the decay term $\eta  (1 -u) v$ describing the fade out of the immune response in absence of any virus.

The growth rate $w$ in \cite{kerl_reaction_2012} describes a logistic growth of the virus with a strong Allee effect \cite{allee_principles_1949}, i.\,e.\
\begin{align}\label{eq:growth}
w(u)= (1-u) \frac{u- u_\mathrm{min}}{u+ \kappa}\;\mbox{ with }\; 0<u_\mathrm{min}\ll 1\;\mbox{ and }\;\kappa>0.
\end{align}
The minimal density for the survival of the virus is $u_\mathrm{min}$. 
Otherwise, the virus is locally attacked and it decreases without the secondary immune response from the lymphs. 
The parameter $\kappa$ is a small parameter fitting the growth in Eq.~(\ref{eq:growth}) to a pure logistic growth for values $u$ close to 1.

As usual in population dynamics models, the reaction functions in Eq.~(\ref{eq:sys}) contain terms with opposite monotonicity behavior. 
The growth term $u w(u)$ and the decay term $- \gamma u v$ act oppositely for $u$ in the equation for $u_{,t}$ just like the inflow term $j[u]$ and the decay term $- \eta (1-u) v$ do for $v_{,t}$. 

\begin{remark}\label{rem:zuw}
The particular choice of the growth rate makes $w(u_\mathrm{min})=0$ and $w(1)=0$, and it is positive between the zeros. 
Furthermore $w$ behaves asymptotically like $1-u$ for large $u$, and we find that $w$ is increasing in the interval $[0,u_\mathrm{min}]$. 
Thus, the minimal value $w(u)$ for $u\in[0,1]$ is $w(0)=-u_\mathrm{min}/\kappa$.
\end{remark}

Opposite to the classical Lotka-Volterra model, the Allee effect allows a population to become extinct.
Besides, the Allee effect does not influence qualitatively the system behavior for larger values $u$.

\begin{remark}\label{rem:zuw2}
Eq.~(\ref{eq:growth}) norms the capacity of the logistic growth to 1 because $w(u) <0$ for all $u>1$. 
There is no loss of generality because the normalization of $u$ is a pure scaling. 
A possible $u$ with $u(t, \mathbf{x}) >1$ at some $\mathbf{x}$ decays in finite time below 1.
Due to this realistic property of the model, system~(\ref{eq:sys}) is suitable only for $u(t, \mathbf{x}) \leq1$. 
\end{remark}

The non-local inflow term is
\begin{align}\label{eq:inflow}
j[u](\mathbf{x})= \delta \chi_{\Theta}(\mathbf{x}) \int_\Omega u(t, \mathbf{x}) \,  \mathrm{d} \mathbf{x} = \delta  \chi_{\Theta}(\mathbf{x}) U(t) \;\mbox{ where }\;U(t)=\int_\Omega u(t,\mathbf{x}) \,  \mathrm{d} \mathbf{x}
\end{align}
is the total amount of virus, and $\chi_{\Theta}(\mathbf{x})$ is a non-negative function with $\mathrm{supp}\, \chi_\Theta (\mathbf{x}) = \Theta \subset \Omega$ and 
\begin{align}\label{eq:chi1}
\int_\Omega \chi_{\Theta}(\mathbf{x}) \, \mathrm{d} \mathbf{x}= \int_\Theta \chi_{\Theta}(\mathbf{x}) \, \mathrm{d} \mathbf{x}=1.
\end{align}
As a realistic inflow, we consider $\chi_\Theta$ to be at least a bounded and piecewise continuous function. 

A non-smooth example for $\chi_\Theta$ is the characteristic function on the subdomain $\Theta \subset \Omega$ providing
$1/ \lvert \Theta \rvert$ for ${\mathbf{x}}\in\Theta$ and $0$ elsewhere. 

Analogously, to Eq.~(\ref{eq:inflow}), we define the integral of the non-negative $v$ over~$\Omega$ as
\begin{align}\label{eq:V}
V(t)= \lVert v(t,\mathbf{x})\rVert_{L^1(\Omega)} = \int_\Omega v(t, \cdot)   \,\mathrm{d }\mathbf{x}.
\end{align}

This expression gives the total amount of T~cells in $\Omega$ and is important for the harm of an infected organism.

\begin{remark}\label{rem:zuj}
Since the integral over $\chi_{\Theta}(\mathbf{x})$ is $1$, we see that the total inflow of T~cells
\begin{align}\label{eq:chi1lgm}
J=\int_\Omega j[u]({\mathbf{x}})\, \mathrm{d} \mathbf{x} = \int_\Theta \delta \chi_\Theta (\mathbf{x})  U(t) \, \mathrm{d} \mathbf{x}=\delta U(t) 
\end{align}
is proportional to the total amount of virus.
\end{remark}

This property of the model reflects the virus-depending strength of the immune response. 
The proportionality in Eq.~(\ref{eq:chi1lgm}) contains the monotonous increase of strength of the immune system when the total amount of virus increases. 

The total amount $U(t)$ of virus at the time instant $t$ occurs in Eq.~(\ref{eq:inflow}) and results in the non-local inflow term in the reaction diffusion system in Eq.~(\ref{eq:sys}). 
Consequently, the model equations in Eq.~(\ref{eq:sys}) are only meaningful if the integral in Eq.~(\ref{eq:inflow}) exists and is finite, i.\,e.\, if $u(t,\cdot)\in L^1(\Omega)$. 
We show in Sec.~\ref{sec:lpbounds}, that the solutions $u$ and $v$ stay in $L^1(\Omega)$ after they are once in $L^1(\Omega)$. 
So in particular, we show therewith that no blow-up in $L^1(\Omega)$ will occur, cf.\ Sec.~\ref{sec:lpbounds}.
These results will imply that both, the total amount of virus and T~cells are bounded in the model. 

For this investigation, we have a closer look on the mechanisms in model (\ref{eq:sys}). 
The reaction terms in system~(\ref{eq:sys}) contain oppositely acting mechanisms. 
For $u$, the growth $u w(u)$ leads to an increase of $u$ for $u \in (u_\mathrm{min},1)$. 
As an opposite effect, the term $- \gamma uv$ describes a decrease depending on $v$. 
The equations for $v_{,t}$ contains three mechanisms.
First, $v$ increases with the total amount of $u$ in the domain $\Omega$. 
The increase of~$v$ is space-depending and takes place in a subdomain $\Theta \subset \Omega$.
The second mechanism is a decrease $- \eta v$, which depends linearly on $v$. 
As a third mechanism, the term $\eta uv$ corresponds to $- \gamma uv$ in the first equation, compare the classical Lotka Volterra system. 

Fig.~\ref{fig:statediagram} shows a state chart of the local reaction mechanisms. 
It is simplified and abstracts from the space dependency of the increase of $v$ by the inflow term $j[u]$.

\begin{figure}[tbh]
\begin{center}
\includegraphics[width= 0.9\textwidth]{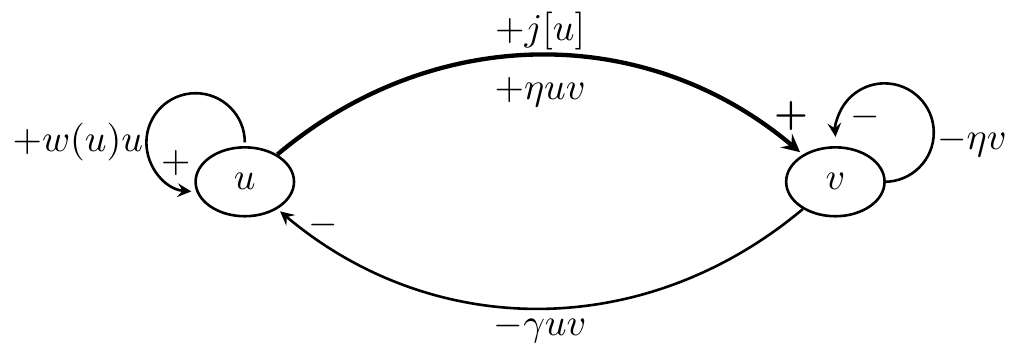}
\caption{State chart for the reaction mechanisms of system~(\ref{eq:sys}) for realistic $u  \in (u_\mathrm{min},1)$. The influence of $u$ is positive on both populations $u$ and $v$. In contrast, the influence of $v$ on both populations is negative. Additional to the dynamics of the classical predator-prey model, there is a positive influence on $v$ just depending on $u$, compare the thicker line. This might lead to an unbounded growth of $v$, what is part of our discussion. 
}
\label{fig:statediagram}
\end{center}
\end{figure}

The non-local inflow term $j[u]$ is a considerate expansion of the classical Lotka Volterra system because the growth of the predator depends directly on the prey in Eq.~(\ref{eq:sys}). 
That enforces the feedback loop in the way, that an increasing predator population slows down its own growth by diminishing the prey population in $u$, compare $(-)$ in Fig.~\ref{fig:statediagram}.

The interplay of oppositely acting mechanisms leads to interesting solutions. 
We observe in \cite{kerl_reaction_2012} that the system~(\ref{eq:sys}) has solutions which can be divided into two qualitative different types. 
On the one hand, there are solutions tending towards zero.
On the other hand, we find solutions with a tendency towards a stationary state which is spatially inhomogeneous. 
The used parameters and the shape and size of the domain $\Omega$ control towards which stationary state the solution is tending.
See~\cite{kerl_reaction_2012, reisch_chemotactic_2019, reisch_entropy_2020, reisch_reaktions-diffusions-gleichungen_2020} for further details on the analytical results.  

\begin{figure}[hbt]
\begin{center}
\includegraphics[width=0.95\textwidth]{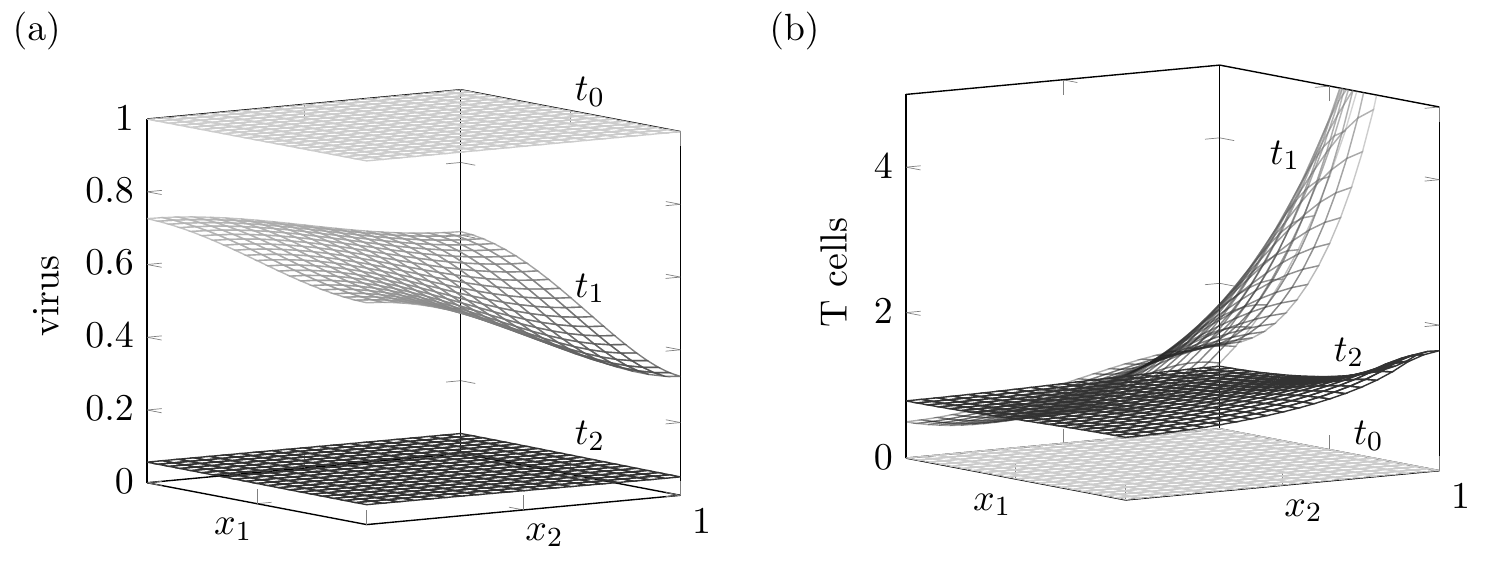}
\caption{Numerical simulation with a solution interpreted as healing infection course. Used parameter values are $u_\mathrm{min}=0.05$, $\kappa=0.01$, $\gamma=0.9$, $\delta=3.7$, $\eta=0.2$, $\alpha=0.6$ and $\beta=0.3$.
The initial conditions for $t_0=0$ (bright mesh) show the amount of virus and T cells right after the activation of the immune response. 
In (b) T~cells enter the domain through an area $\Theta$ around $(x_1, x_2)=(1,1)$.
The virus is killed by the T~cells and decays for $t_1=0.75$ and $t_2=3$ (dark mesh). The amount of T~cells reduces due to the very low virus concentration. Both population vanish after an active phase. }
\label{fig:healing}
\end{center}
\end{figure}

As the model was found in the context of liver infections, we interpret the two qualitative different solution types as different infection courses. 
Solutions with a tendency towards zero are associated with healing courses, see Fig.~\ref{fig:healing}. 
The immune system is able to kill all infected cells during an active phase and therefore, the virus vanishes.
Afterwards, the immune reaction fades out and the T~cell amount tends towards zero as well, see Fig.~\ref{fig:healing}(f).

\begin{figure}[hbt]
\begin{center}
\includegraphics[width=0.95\textwidth]{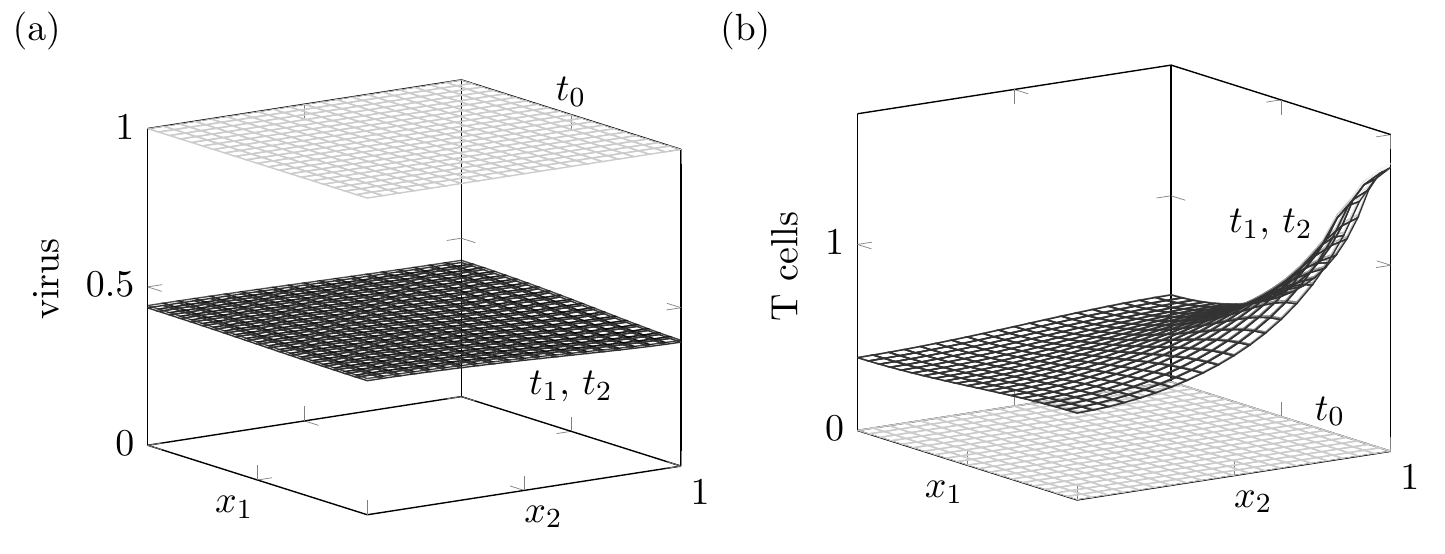}
\caption{Numerical simulation with a solution interpreted as chronic or persisting infection course. Used parameter values are the same as in Fig.~\ref{fig:healing} but $\delta=0.7$ and $\eta=0.9$. The time steps are $t_0=0$ (bright), $t_1 = 10.5$, $t_2=30$ (dark).
Starting with the same initial conditions as in Fig.~\ref{fig:healing}, the virus and the T~cells persist in the whole domain. The T~cell amount is higher around the portal field $\Theta$. There is nearly no difference between the spread at $t_1$ and at $t_2$.}
\label{fig:chronic}
\end{center}
\end{figure}

Solutions with tendency towards stationary spatially inhomogeneous states are interpreted as persisting or chronic infections, compare Fig.~\ref{fig:chronic}. 
After an active phase with a strong immune reaction in Fig.~\ref{fig:chronic}(d), the T~cell amount decays, but does not vanish and the virus persists in the liver. 
In the stationary phase, there is still virus in the whole domain $\Omega$, see Fig.~\ref{fig:chronic}(e), and T~cells as well, see Fig.~\ref{fig:chronic}(f). 

In addition to Fig.~\ref{fig:healing} and~\ref{fig:chronic}, where space-dependent solutions for a fixed time are displayed, Fig.~\ref{fig:sumplot} shows the trajectories of the total virus $U(t)$ and T~cell populations $V(t)$ of different infection courses over the time. 

\begin{figure}[bht]
\begin{center}
\includegraphics[width=\textwidth]{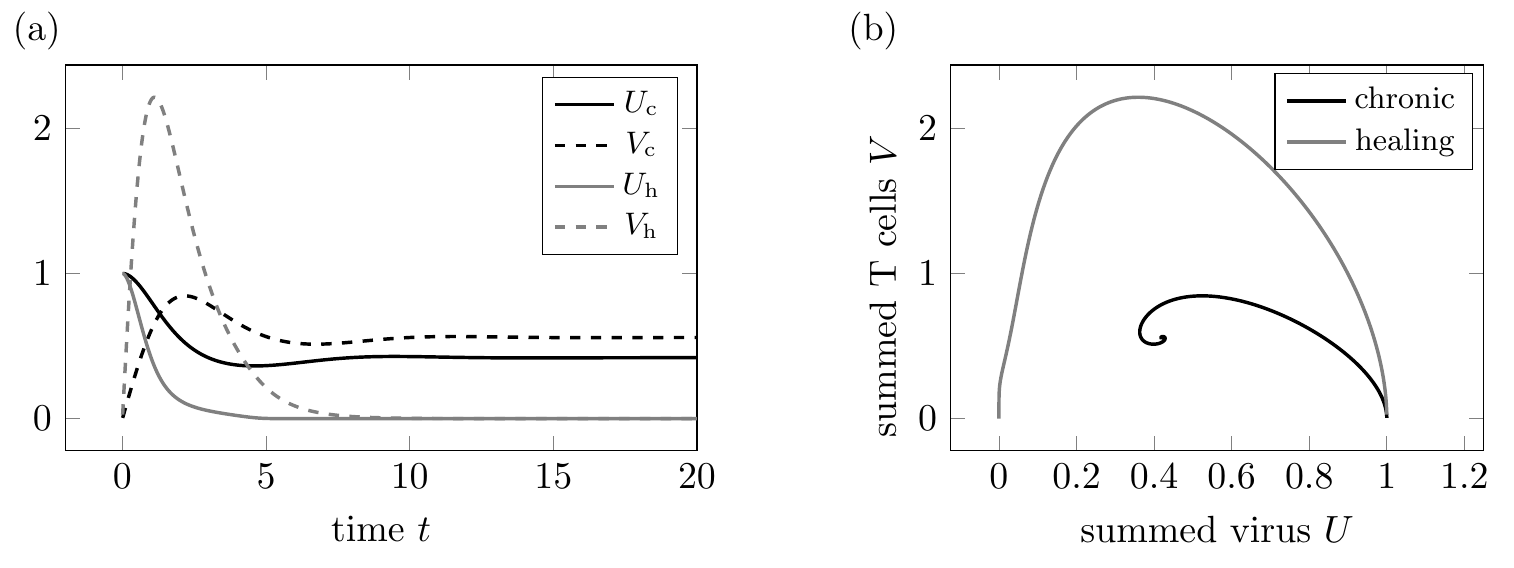}
\caption{Numerical simulations according to those in Fig.~\ref{fig:healing} and Fig.~\ref{fig:chronic}.
(a) Total virus~$U$ and T~cell amount $V$ during a chronic or healing infection course over the time. (b) Summed dynamics of a healing or chronic infection course in phase space. }
\label{fig:sumplot}
\end{center}
\end{figure}

Fig.~\ref{fig:sumplot} shows, that the total populations tend towards a stationary state in both cases. 
Together with the space dependent Fig.~\ref{fig:chronic}, Fig.~\ref{fig:sumplot} shows the tendency of the solution towards a spatially inhomogeneous stationary distribution for a chronic infection course.

\section{Existence}\label{sec:existence}

The model in Eq.~(\ref{eq:sys}) reflects biological structures, see Remark~\ref{rem:scale}, and uses a non-local and space-dependent term for modeling the biological structure of the application. 
The analysis of this model leads to an interesting new problem which cannot be handled easily by standard approaches. 
Besides, we are interested in a proof using interpretable intermediate steps for gaining a deeper understanding of the systems dynamics. 

Of course, there are many theorems for the existence of a solution of a reaction diffusion system or more general a parabolic partial differential equation.
In this section, we mention some important results on the existence of solutions for reaction diffusion equations and explain, why they cannot be applied directly to the system~(\ref{eq:sys}). 

There are at least two main approaches often used in proofs of existence theorems for parabolic partial differential equations. 
One approach uses fixed point theorems, like the Banach fixed point theorem, the Brouwer fixed point theorem and from this following the Schauder and the Leray-Schauder fixed point theorems, \cite{evans_partial_2010}. 
The second approach uses semigroup theory, see \cite{amann_linear_2019, lunardi_analytic_2012}. 

The first approach using fixed point theorems can be found for example in \cite[p. 536]{evans_partial_2010}. 
There, the existence and uniqueness of solutions is shown under the requirement, that the local reaction function $\mathbf{f}$ is Lipschitz continuous with respect to $\mathbf{q}=(u,v)^\mathrm{T}$. 
This requirement is used for showing the contraction of the operator for the fixed point theorem. 
Additionally, the theorem in \cite{evans_partial_2010} requires Dirichlet boundary conditions. 

In \cite[p. 188]{roubicek_nonlinear_2013}, an existence theorem for a reaction diffusion system with Lotka Volterra reaction terms is shown. 
The proof is based on the Schauder fixed point theorem and uses a-priori bounds for the state variables. 

There are several proofs for monotonous reaction functions as well, see~\cite[p. 120]{showalter_monotone_1997}.

Unfortunately, the reaction functions in Eq.~(\ref{eq:sys}) are neither globally Lip\-schitz continuous with respect to $u$ and $v$, nor monotonous. 
Even if $u$ is bounded by construction by an upper limit 1, an a-priori upper bound for $v$ is not obvious. 
We show the existence of a global upper bound in Sec.~\ref{sec:linfty}. 

Existence results using a semigroup approach are based on limited growth conditions, for example \cite[p. 276]{lunardi_analytic_2012} or \cite[p. 75]{henry_geometric_1981}.
Due to the non-local integral term, the nonlinear terms and the unavailable a-priori bound for $v$, the system in Eq.~(\ref{eq:sys}) does not fulfill the requirements for these existence results. 
As already mentioned, the existence of a finite a-priori bound for $v$ and therefore the boundedness of $v$ in the $L^\infty$-norm is a relevant question concerning the application in modeling liver infections. 

Results for reaction diffusion systems with non-local effects can be divided into results for nonlinear diffusion and nonlinear reaction terms. 
The global existence of solutions for systems with nonlinear diffusion,
\begin{align*}
	u_{i,t} =  f_i (u_1, \dots u_m)+ \Delta \varphi_i (u_i) 
\end{align*}
with homogeneous Dirichlet boundary conditions is shown in \cite{laamri_global_2017}. 
The results yield if the solutions are non-negative and the total mass is controlled. 
Additionally, an a-priori estimate in the $L^1(\Omega)$-norm for the nonlinear reaction functions is required. 

In \cite{rouchon_universal_2003}, the reaction diffusion equation
\begin{align*}
	u_{,t} = \Delta u + \int_\Omega u^p \, \mathrm{d} y
\end{align*}
with a non-local term and with homogeneous Dirichlet boundary conditions is analyzed. 
The global existence of non-negative solutions is shown for any $p>1$. 

As a third example, the global existence of solutions of the general formulation
\begin{align*}
	u_{,t} + A(u) = F(u),
\end{align*}
where $A$ is a parabolic operator and $F$ is bounded in the $L^2(\Omega)$-norm is shown in \cite{anguiano_asymptotic_2010}. 

The results are mainly for single equations instead of systems, and the requirements are not fulfilled for system~(\ref{eq:sys}). 
Again, the system with coupled equations and an integral term require new approaches for proving the existence of globally bounded solutions. 

The adaption of the named existence theorems on our system~\ref{eq:sys} requires - if possible at all - severe modifications on a technical mathematical level.
However, by proving the longterm existence, we aim to develop a deeper understanding of the infection application. 
Therefore, we present a step by step proof and accompany it by biological and medical applications.

Now, we show the existence of solutions and their boundedness in $L^\infty(\Omega)$, which allows a point-wise estimation of the maximal virus and T~cell amounts. 
The section has the following structure. 
First, the existence of a weak solution for a small time span $[0,T)$ is shown. 
We discuss some basic properties of such solutions like non-negativity of $u$ and $v$ and boundedness of $u$. 
These properties are important for modeling purposes as negative values are not interpretable in the context of densities of virus and T~cells. 

In Sec.~\ref{sec:lpbounds}, the boundedness of $v$ in $L^1(\Omega)$ is shown.
This result shows a boundedness of the total amount of virus and T~cells at a certain time. 
Afterwards and building up on this result, the boundedness of the norm $\lVert v \rVert_{L^2(\Omega)}$ is proven. 

Finally, in Sec.~\ref{sec:linfty} the boundedness of the norm $\lVert v \rVert_{L^\infty(\Omega)}$ is shown and using it, the global longterm existence of weak solutions of Eq.~(\ref{eq:sys}) is shown.

\subsection{Properties of the weak solution}

Starting with the definition of a weak solution, the existence of a weak solution of Eq.~(\ref{eq:sys}) for a small time span $[0,T)$ is shown.

\begin{definition}
A weak solution of (\ref{eq:sys}) on the time-interval $[0,T)$ is a pair of functions $(u,v)$ with $u,v\in L^2([0,T);H^1(\Omega))$ and 
$u_{,t},v_{,t}\in L^2([0,T);H^{-1}(\Omega))$ for which
\begin{align*}
\begin{aligned}
\int_\Omega u_{,t} \varphi \, \mathrm{d} \mathbf{x}&=\int_\Omega  u w(u) \varphi - \gamma u v\varphi  - \alpha \nabla u \cdot \nabla \varphi \, \mathrm{d} \mathbf{x} , \\
\int_\Omega v_{,t} \varphi \, \mathrm{d} \mathbf{x}&=\int_\Omega  j[u] \varphi- \eta  (1 -u) v\varphi    - \beta \nabla v \cdot \nabla\varphi \, \mathrm{d} \mathbf{x}
\end{aligned}
\end{align*}
is fulfilled for all $\varphi \in H^1(\Omega)$ with $\varphi=\varphi(\mathbf{x})$ and almost every time $t\in[0,T)$.
\end{definition}

In \cite[Theorem 9.2.2, p. 536]{evans_partial_2010} the existence of a unique weak solution of a reaction diffusion system 
\begin{align*}
\dot{\mathbf{q}} = \mathbf{f}(\mathbf{q}) + D \Delta \mathbf{q} & \quad \text{with } \mathbf{q}= \begin{pmatrix} u \\ v \end{pmatrix}, \quad D = \begin{pmatrix}  \alpha & 0 \\ 0 &  \beta \end{pmatrix}
\end{align*}
with Dirichlet boundary conditions and a Lipschitz continuous reaction function $\mathbf{f}$ with respect to $\mathbf{q}$ is proven. 
The first step of the proof shows the existence of a weak solution in case of an externally given function $\mathbf{h}(t)= \mathbf{f}(\mathbf{q}(t))$ replacing the reaction terms. 
Additionally, it is shown in this step, that the time derivative of the solution is a $L^2(\Omega)$-function as well.
Even if the reaction diffusion system (\ref{eq:sys}) has Neumann boundary conditions and the reaction function does not fulfill global Lipschitz conditions with respect to the state variables, this step is adaptable by the following considerations.

Regarding a solution $\mathbf{q}  \in C([0,T); L^2(\Omega, \mathbb{R}^2))$, which is bounded in a suitable chosen time interval $t \in [0,T)$. 
Define $\mathbf{h}(t)= \mathbf{f}(\mathbf{q}(t))$ as a right-hand side for the general parabolic system
\begin{align}\label{eq:parabolic}
    \begin{aligned}
        \mathbf{q}_{,t} - D \Delta \mathbf{q} &=  \mathbf{h}(t)  \ \ && \text{ for } \mathbf{x} \in \Omega, 0<t < T, \\
        \nabla \mathbf{q} \cdot \mathbf{n} & = \mathbf{0}  \ \ && \text{ for } \mathbf{x} \in \partial \Omega, 0<t < T, \\
         \mathbf{q} (0,\mathbf{x}) & = \mathbf{q}_0(\mathbf{x})  \ \ && \text{ for } \mathbf{x} \in  \Omega . \\
     \end{aligned}
\end{align}
Due to the boundedness of $\mathbf{q}$ in the limited time interval and the smoothness of $\mathbf{f}$, the function $\mathbf{h}$ is regular in the sense, that $\mathbf{h} \in L^2( [0,T); L^2(\Omega, \mathbb{R}^2))$.

In \cite[Theorem 3, p. 378]{evans_partial_2010}, the existence of a weak solution for systems like in Eq.~(\ref{eq:parabolic}) but with Dirichlet boundary conditions is shown. 
By replacing the Sobolev space $H^1_0(\Omega)$ by $H^1(\Omega)$ and changing some of the constants, a completely analogous proof assures the existence of a weak solution in case of homogeneous Neumann boundary conditions.

\begin{theorem}\label{ex_loc}
Let $u_0 , v_0 \in L^\infty (\Omega)$ and $T>0$ such that $u(t), v(t) \in L^\infty (\Omega)$ for $t \in [0,T)$. Then, there exists a weak solution $u(t,\cdot), v(t, \cdot) \in L^2((0,T), H^1(\Omega))$ with $u_{,t}(t,\cdot), v_{,t}(t, \cdot) \in L^2((0,T), H^{-1}(\Omega))$ of system~(\ref{eq:sys}). 
\end{theorem}

The proof follows \cite[Theorem 9.2.2, p. 536]{evans_partial_2010} and \cite[Theorem 3, p. 378]{evans_partial_2010} with the mentioned adaptions of the boundary conditions. 

The interpretation of the theorem fits with the observation in nature: In a finite time, there cannot be an infinite amount of virus or T~cells at a single point. 
The maximal amount of virus and T~cells is bounded for a time interval $[0,t)$. 

The weak solution $(u,v)$ of Eq.~(\ref{eq:sys}) fulfills some basic properties. 

\begin{lemma}[Non-negativity]\label{lem:posi}•
If $u_0(\mathbf{x}) \geq 0 $ and $v_0(\mathbf{x})\geq0$ for all $\mathbf{x} \in \Omega$, then $u(t,\mathbf{x}) \geq 0$ and $v(t,\mathbf{x}) \geq 0$ yield for all $t \in(0,T)$ and all $\mathbf{x} \in \Omega$. 
\end{lemma}
\begin{proof} 
Regard the point $\mathbf{x}_\mathrm{min} \in \Omega$, where one state variable has its minimal value. 
If at one time $t$ the minimum $\displaystyle \min_{\mathbf{x}\in\Omega} u(t,\mathbf{x})=u(t,\mathbf{x}_{\mathrm{min}})=0$ touches the lower bound of the positive domain, then Eq.~(\ref{eq:sys}) provides that the reaction term 
\begin{align*}
u(t,\mathbf{x}_{\mathrm{min}})w(u(t,\mathbf{x}_{\mathrm{min}}))-\gamma u(t,\mathbf{x}_{\mathrm{min}})v(t,\mathbf{x}_{\mathrm{min}})=0
\end{align*}
 vanishes at the point $\mathbf{x}_{\mathrm{min}}$ of the minimum. 
 At the same time, $\Delta u(t,\mathbf{x}_{\mathrm{min}})\ge 0$ at this point. 
 Thus, $u$ cannot pass zeros, and stays non-negative.
 
Similarly, if $\displaystyle \min_{\mathbf{x}\in\Omega} v(t,\mathbf{x})=v(t,\mathbf{x}_{\mathrm{min}})=0$, then $j[u](\mathbf{x}_{\mathrm{min}})\ge 0$, $\eta (1-u(t,\mathbf{x}_{\mathrm{min}}))v(t,\mathbf{x}_{\mathrm{min}})=0$ and $\Delta v(t,\mathbf{x}_{\mathrm{min}})\ge 0$, and $v$ stays non-negative as long as it exists, too.
\end{proof}

Consequently, the model reflects the fact that the amount of virus $u$ and of T~cells $v$ are non-negative. 

Next, we proof that the amount of virus is bounded point-wise. 

\begin{lemma}[Boundedness of $u$]\label{lem:boundu}
If $u_0(\mathbf{x})\leq 1 $ for all $\mathbf{x} \in \Omega$, then the (weak) solution $u(t,\mathbf{x})$ is bounded by $u(t,\mathbf{x}) \leq 1$ for all $t \in (0,T) $ and $\mathbf{x} \in \Omega$. 
\end{lemma}
\begin{proof}
Regard the maximum value $\displaystyle \max_{\mathbf{x}\in\Omega} u(t,\mathbf{x})=u(t,\mathbf{x}_{\mathrm{max}})$. 
If this maximum is equal to $1$, then the growth term $uw(u)$ vanishes at $\mathbf{x}_{\mathrm{max}}$, see Remark~\ref{rem:zuw}, the predator term $-\gamma u v$ is not positive, and the diffusion term $\alpha\Delta u$ is not positive, too, cf.\ proof of Lemma~\ref{lem:posi}. 
Consequently, the maximum $\displaystyle \max_{\mathbf{x}\in\Omega} u(t,\mathbf{x})$ cannot grow above $1$.
\end{proof}

According to Remark~\ref{rem:zuw2}, the model in Eq.~(\ref{eq:sys}) is not suitable for values of $u$ larger than $1$. 
Initial conditions with $u_0>1$ does not affect the boundedness of $u$ by $1$.

\begin{corollary}\label{cor:durcheins}
If $u_0(\mathbf{x})>1 $ for some $\mathbf{x} \in \Omega$, there exists a time $t_1$ with $u(t, \mathbf{x}) \leq 1$ for all $t> t_1$.
\end{corollary}

\begin{proof}
Again, we regard the maximum $\displaystyle \max_{\mathbf{x}\in\Omega} u(t,\mathbf{x})=u(t,\mathbf{x}_{\mathrm{max}})$. 
If it is larger than $1$, the logistic growth $uw(u)$ is strictly negative at the point $\mathbf{x}_{\mathrm{max}}$. 
Since $u$ has its maximum at $\mathbf{x}_{\mathrm{max}}$, the diffusion term fulfills $\alpha\Delta u\le 0$. 
At the same time, $v$ is increasing, so that the predator term $-\gamma u v$ is larger than $0$, and the maximum $\displaystyle \max_{\mathbf{x}\in\Omega} u(t,\mathbf{x})$ passes the value $1$ with a non-zero time derivative at a finite time instant $t_1$.
\end{proof}

The proof shows that $u$ does not tend to $1$, but rather passes $1$. <
That means that the virus decays under its capacity, whenever an active immune response exists.
We formulate this observation in a next corollary saying that $u$ becomes smaller than $1$ together with a non-vanishing $v$ on some sub-domain of~$\Omega$.

\begin{corollary}\label{cor:wegvoneins}
All bounded and non-vanishing initial values allow to find a time instant $t_2$ for which $u(t_2,{\mathbf{x}})\le 1$ holds true for all 
${\mathbf{x}}\in\Omega$ and $U(t_2)<(1- \varepsilon)\lvert \Omega \rvert$ yields for $\varepsilon>0$. 
At the same time, $v$ is not vanishing at points $\mathbf{x} \in \Omega$ where $u$ is smaller than $1$.
\end{corollary}

\begin{proof} 
If the function $u(t,\cdot)$ is not identical to $\displaystyle \max_{\mathbf{x}\in\Omega} u(t,\mathbf{x})$, then Cor.~\ref{cor:durcheins} provides $t_1$ with $u(t,\mathbf{x}) \leq 1 $ for all $t>t_1$.
Due to the position-dependency of $u$, there exists $\varepsilon > 0$ with $U(t_1) < (1- \varepsilon) \lvert \Omega \rvert$.

For space-independent functions $\displaystyle u(t_1,\mathbf{x})=\max_{\mathbf{x}\in\Omega} u(t_1,\mathbf{x})$, only the situation 
\begin{align*}
u(t_1,\mathbf{x})=\max_{\mathbf{x}\in\Omega} u(t_1,\mathbf{x})=1
\end{align*}
is interesting. 
If $u(t, \mathbf{x}) > 1$ for all $\mathbf{x} \in \Omega$, the reaction terms lead to a decay with $u(t_2, \mathbf{x}) \leq1$ due to Cor.~\ref{cor:durcheins}.
Let $\displaystyle (t_1,\mathbf{x})=\mathrm{argmax}_{\mathbf{x}\in\Omega} u(t_1,\mathbf{x})=1$.

Since the inflow $j[u]$ is positive in both cases, $v$ increases, and the predator term $-\gamma u v$ is strictly negative in $\Theta$ for all $t>t_1$.
Therefore the assertions are fulfilled for every $t_2>t_1$ with sufficiently small $t_2-t_1$.
\end{proof}

In the following, we assume initial conditions $u(0,\mathbf{x}) \leq 1$ for all $\mathbf{x} \in \Omega$. 
As shown in Cor.~\ref{cor:durcheins} and~\ref{cor:wegvoneins} and according to the formulation of system~(\ref{eq:sys}), this is not a restriction. 

We interpret Cor.~\ref{cor:wegvoneins} in the  light of application. 
Even if there would be a higher amount of virus than the upper limit allows, the additional virus vanish by a negative growth term and spread out by diffusion. 
The negative growth can be interpreted as a decay due to a limited number of free liver cells where the virus can attach.

\begin{corollary}\label{cor:Ubeschr}
Let $(u, v)$ be a solution of system~(\ref{eq:sys}) for $t \in (0,T)$ with initial conditions $0 \leq u(0,\mathbf{x}) \leq 1$ and $0 \leq v(0, \mathbf{x}) \leq v_\mathrm{max} < \infty$.
Then, the $L^1(\Omega)$-norm $U(t) =\lVert u(t, \cdot) \rVert_{L^1(\Omega)} $ is bounded by $U(t)  \leq \lvert \Omega \rvert $ for all times $t \in [0,T)$.
\end{corollary}

\begin{proof}
Due to Lemma~\ref{lem:boundu}, the solution $u(t,\mathbf{x})$ is bounded by 1. 
Integration of both sides of $u \leq 1$ gives $U(t) \leq \lvert \Omega \rvert$.
\end{proof}

With these results, we found a (weak) solution for a time interval $[0,T)$, which is non-negative and at least one component of the solution, namely $u$, is bounded. 
The increase of the second component $v$ depends on the $L^1(\Omega)$-norm $U$ of $u$. 
Hence, until now, $v$ could still grow over all bounds.

Consequently, we have to show that the increase of $v$ happens simultaneously to a decrease of $U$, cf. Fig.~\ref{fig:statediagram}, and that this simultaneity makes $v$ to be bounded in the different norms. 

Since we will need it in the next section for showing that blow-ups of the solution of system~(\ref{eq:sys}) do not occur, we prove that $u$ is not only bounded by $1$ but it is sufficiently remote from $1$ after some time. 
The medical background suggests that a virus density close to $1$ provokes an increase of the immune response. 
Hence, the virus density decreases. 
This slows down the influx of T~cells again, compare the opposite directions of the mechanisms in Fig.~\ref{fig:statediagram}.
The following Lemma~\ref{lem:integral} will give a very rough estimate for this observation. 

But first, we consider the solution $v_\mathrm{aux} = v_\mathrm{aux}(\mathbf{x})$ of the auxiliary stationary problem
\begin{align}\label{eq:aux}
\begin{aligned}
- \beta \Delta v + \eta v& = \chi_\Theta (\mathbf{x})   && \text{ for } \mathbf{x} \in \Omega ,\\
 \nabla v \cdot \mathbf{n} &=0 && \text{ for } \mathbf{x} \in \partial \Omega .
\end{aligned}
\end{align}
The function $\chi_\Theta (\mathbf{x} ) \geq 0 $ is at least piecewise continuous and not vanishing in the whole domain $\Omega$.
Consequently $v_\mathrm{aux}$ is continuous, bounded and positive. 
Since $\chi_\Theta (\mathbf{x} )$ is positive only in the influx region~$\Theta$, there is some value $v_\mathrm{thr} >0 $ with $v_\mathrm{aux}(\mathbf{x}) \geq v_\mathrm{thr}$ for all $\mathbf{x} \in \Theta$.

Therewith, we are prepared to prove the announced Lemma. 

\begin{lemma}\label{lem:integral}
Let $u, v$ be weak solutions of (\ref{eq:sys}). 
For all $\varrho \geq 0 $, there is a $\theta$ with $ 0< \theta<1$ and a time $t_3$ with 
\begin{align}\label{eq:intrho}
\int_\Omega uv^\varrho \, \mathrm{d} \mathbf{x}  \leq \theta \int_\Omega v^\varrho \, \mathrm{d} \mathbf{x}
\end{align}
for all $ t \geq t_3$.
\end{lemma}
\begin{proof}
First, we show that there is at least one $t_3$ for which Eq.~(\ref{eq:intrho}) is fulfilled.

Assume, there would be no such $t_3$.
Then, $u$ must be equal $1$ almost everywhere in $\mathrm{supp} \, v \subset \Omega$ for all time $t$.
As a solution of Eq.~(\ref{eq:sys}), $u$ is continuous with respect to $\mathbf{x}$.
Consequently, $u$ must be equal $1$ in $\Theta$ and we get the rough estimate $ U(t) \geq \lvert \Theta \rvert $.
Now, the evolution of $v$ in Eq.~(\ref{eq:sys}) reads
\begin{align*}
v_{,t }= j[u] - \eta (1-u) v + \beta \Delta v \geq \delta U(t) \chi_\Theta (\mathbf{x}) - \eta v + \beta \Delta v,
\end{align*}
and after a transient phase, we get
\begin{align*}
v(t, \mathbf{x}) \geq \delta U(t) v_\mathrm{aux} (\mathbf{x})  \geq \delta \lvert \Theta \rvert v_\mathrm{aux} (\mathbf{x})
\end{align*}
and thus
\begin{align}\label{eq:vthr}
v(t, \mathbf{x}) &\geq \delta \lvert \Theta \rvert v_\mathrm{thr}  \quad \quad \text{ for all } \mathbf{x} \in \Theta .
\end{align}
Finally, the first equation in system~(\ref{eq:sys}) reads 
\begin{align*}
u_{,t}= u w(u) - \gamma u v + \alpha \Delta u \leq u \left( w(u) - \gamma \delta U(t) v_\mathrm{aux} ( \mathbf{x}) \right ) + \alpha \Delta u .
\end{align*}
So, Eq.~(\ref{eq:vthr}) implies $u_{,t}  <0 $ for all $\mathbf{x} \in \Theta$, what contradicts the assumption $u= 1$ in $\Theta$. 
Consequently, there is at least one time instant $t_3$ fulfilling Eq.~(\ref{eq:intrho}). 

If we now assume that $u$ grows again after $t_3$ so that the estimate (\ref{eq:intrho}) is hurt for every $\theta <1$ at some $t_4$, that would mean $u$ gets arbitrarily close to $1$ in $\Theta$.
This is again a contradiction to 
\begin{align*}
w(u) - \gamma \delta U(t) v_\mathrm{aux} (\mathbf{x})  \leq w(u) - \gamma \delta \lvert \Theta \rvert v_\mathrm{aux} (\mathbf{x}) <0
\end{align*}
at this time instant $t_4$.
\end{proof}

In the next steps, we show, that there exists an upper bound for $v$ as well. 
First, we show, that $v$ is bounded in $L^1(\Omega)$ for $t \in (0,T)$. 
Next, we expand this property for all times $T>0$.
As an intermediate step, we show $v(t, \cdot) \in L^2(\Omega)$. 
Finally, by using the stationary solution of another related elliptic equation for a stationary problem, we prove that $v$ is bounded and smooth for all times $t$, 
$v(t, \cdot) \in L^\infty (\Omega)$.

\subsection{$L^p(\Omega)$ bounds}\label{sec:lpbounds}

First, we determine a $L^1(\Omega)$ bound for $v$.
With Theorem~\ref{ex_loc} we have a (weak) solution $(u ,v)$ with $u(t,\cdot), v(t,\cdot) \in H^1(\Omega)$ for $t \in [0, T)$ with a time~$T$. 

In this section, we show, that $V$ is not growing to infinity for $t \in [0,T)$. 

Therefore, we regard the time derivative of the functional
$\Phi=\eta U + \gamma V$ 
which is a linear combination of the $L^1(\Omega)$-norms of $u$ and $v$.
The functional~$\Phi$ can be interpreted as a measure of the total harm of the infection, compare \cite{reisch_modeling_2019}. 

\begin{theorem}\label{thm:sigma} 
Any pair $(U(0),V(0))$ with $0 \leq U(0)\le |\Omega|$ and $0\leq V(0) < \infty$ allows to find a trapezoid  $\Sigma=\{(U,V)\,:\,U\in[0,|\Omega|],V\in[0, V_\mathrm{up}- \frac{\eta}{\gamma} U]\}$ such that $(U(t),V(t))\in\Sigma$ for all $t\in[0,T)$, i.\,e.\ as long as the solution $(u,v)$ of system~(\ref{eq:sys}) exists in $L^1(\Omega)$.
\end{theorem}
This theorem says that the $L^1(\Omega)$-norm of a solution $(u,v)$ of system~(\ref{eq:sys}) stays in a bounded region, namely within the trapezoid $\Sigma$, compare Fig.~\ref{fig:sigma}, as long as a weak solution exists. 

\begin{proof}
The time derivative of the functional $\Phi=\eta U + \gamma V$ is with system~(\ref{eq:sys})
\begin{align*}
\begin{aligned}
\Phi_{,t} &=\frac{\mathrm{d}}{\mathrm{d}t} \int_\Omega \eta u + \gamma v \, \mathrm{d} \mathbf{x} = \int_\Omega \eta u_{,t}(t,\mathbf{x}) + \gamma v_{,t}(t,\mathbf{x}) \, \mathrm{d} \mathbf{x}\\
	&= \int_\Omega  \eta u w(u) - \eta \gamma uv + \eta \alpha \Delta u
	+ \gamma j[u] +  \eta \gamma u v - \eta \gamma v + \gamma \beta \Delta v  \, \mathrm{d} \mathbf{x} \\
	&= \int_\Omega  \eta u w(u)+ \gamma j[u]- \eta \gamma v \, \mathrm{d} \mathbf{x} ,
\end{aligned}
\end{align*} 
where we use the divergence theorem and homogeneous Neumann boundary conditions from Eq.~(\ref{eq:sys}). 

The solution $u$ is meaningful for values between $0$ and $1$.
Due to Cor.~\ref{cor:durcheins}, it suffices to regard solutions $u$ bounded by $0$ and $1$, see Lemma~\ref{lem:posi} and~\ref{lem:boundu}. 
Consequently, the growth function $w(u)$ in Eq.~(\ref{eq:growth}) is smaller than $1$ for all $u \in [0,1]$.  
We get 
\begin{align*}
\begin{aligned}
\Phi_{,t}
	&\leq \eta \int_\Omega u(t,\mathbf{x}) \, \mathrm{d} \mathbf{x}  + \gamma \int_\Omega j[u] \, \mathrm{d} \mathbf{x}	- \gamma \eta \int_\Omega v(t, \mathbf{x}) \, \mathrm{d} \mathbf{x} .
\end{aligned}
\end{align*} 

The inflow term $j[u]$ is bounded, see Remark~\ref{rem:zuj}, and with Eqs.~(\ref{eq:inflow}) and~(\ref{eq:V}), we can write 
 \begin{align}\label{eq:ineqUV}
\Phi_{,t}
	\leq \eta U(t) + \gamma \delta U(t)  - \gamma \eta V(t) =  (\eta+ \gamma \delta) U(t) - \gamma \eta V(t),
\end{align}
which is not positive for $U \leq \lvert \Omega \rvert$ and $V \geq \frac{\eta + \gamma \delta}{\gamma \eta} \lvert \Omega \rvert$.
The derivative is non-positive for
\begin{align*}
\Phi \geq \eta  \lvert \Omega \rvert + \gamma  \frac{\eta + \gamma \delta}{\gamma \eta} \lvert \Omega \rvert = \left( \eta + 1 + \frac{\gamma \delta}{\eta} \right) \lvert \Omega \rvert = \gamma V_\mathrm{up}
\end{align*}
and $U \leq \lvert \Omega \rvert$.

Since $U$ stays lower than $\lvert \Omega \rvert$, compare Cor.~\ref{cor:Ubeschr}, $\Phi \geq \gamma V_\mathrm{up}$ implies $\Phi_{,t} \leq 0 $, see Eq.~(\ref{eq:ineqUV}).
In particular, $\Phi$ cannot pass $\gamma V_\mathrm{up}$ when it is once lower than $\gamma V_\mathrm{up}$ with $U \leq \lvert \Omega \rvert$.
Consequently, the $L^1(\Omega)$-norm $(U,V)$ of a solution $(u,v)$ stays in $\Sigma$ when it starts in $\Sigma$.

If now $\Phi(0)= \eta U(0) + \gamma V(0) \leq \gamma V_\mathrm{up}$, then $\Phi(t) \leq \gamma V_\mathrm{up}$ for all admissible $t$.
If otherwise $\Phi(0) \geq \gamma V_\mathrm{up}$, we have shown that $\Phi$ decreases until $\Phi(t)$ is smaller than $\gamma V_\mathrm{up}$. 

Finally, $\Phi (t) \leq \max \{ \Phi(0), \gamma V_\mathrm{up} \} = \gamma V_\mathrm{up}$ for all admissible $t$ and all initial values allow to construct a suitable $\Sigma$ where the solution stays in.
\end{proof}

In Fig.~\ref{fig:sigma}, the trapezoid $\Sigma$ is shown in the phase space of $(U,V)$. 
The arrows show the direction of the dynamics given by the reaction term in system~(\ref{eq:sys}). 
The arrows of the dynamics point inside $\Sigma$ or at least not to the exterior, especially at the upper bound of $V$. 

\begin{remark}\label{rem:L1v}
Since $\lvert \Omega \rvert >0$, the $L^1(\Omega)$-norm $V(t)$ of $v$ is bounded by
\begin{align}\label{eq:vup12}
V_\mathrm{up} = \frac{1}{\gamma} \left(\eta+ 1+ \frac{\gamma \delta}{\eta} \right) \lvert \Omega \rvert
\end{align}
for all $t \in [0, T)$. 
\end{remark}

That means that the summed strength of the immune response is bounded in a bounded time interval.

\begin{remark}
Let $\Sigma$ be the trapezoid in the phase space $(U,V)$, which is bounded by $U=0$, $V=0$, $U=\lvert \Omega \vert$ and $\eta U+ \gamma V  \leq \gamma V_\mathrm{up}$.
If the $L^1(\Omega)$-norms $U(0)= \lVert u_0(\cdot) \rVert_{L^1(\Omega)} $ and $V(0)= \lVert v_0(\cdot) \rVert_{L^1(\Omega)} $ are inside~$\Sigma$, then the linear combination $\eta U(t) + \gamma V(t)$ with $U(t)$ and $V(t)$ as the $L^1(\Omega)$-norms of the weak solutions $u(t,\mathbf{x}), v(t,\mathbf{x})$ of system~(\ref{eq:sys}) with $t \in (0,T)$ is inside of~$\Sigma$ as well. 

The $L^1(\Omega)$-norms $U(T-\varepsilon )$ and $V(T-\varepsilon )$ are for any $\varepsilon >0$ inside of $\Sigma$.
$\Sigma$ depends only on the initial values, but it is independent of the time $t$ and the solutions $u$ and $v$ theirselves.
Consequently, the $L^1(\Omega)$-norms $U(T )$ and $V(T )$ are inside of $\Sigma$ as well.
The values $u(T, \mathbf{x} )$ and $v(T, \mathbf{x} )$ can be seen as new initial data of system~(\ref{eq:sys}). 
By induction, the  $L^1(\Omega)$-norms $U(t)$ and $V(t)$ are inside of $\Sigma$ for every $t>0$ and the $L^1(\Omega)$-norm $V(t)$ of $v(t,\mathbf{x})$ is bounded by $V_\mathrm{up}$ in Eq.~(\ref{eq:vup12}) for all time $t>0$. 
\end{remark}

This results shows the boundedness of the total amount of T~cells at any time, not only for a limited interval. 

\begin{figure}
\begin{center}
\includegraphics[width=0.65\textwidth]{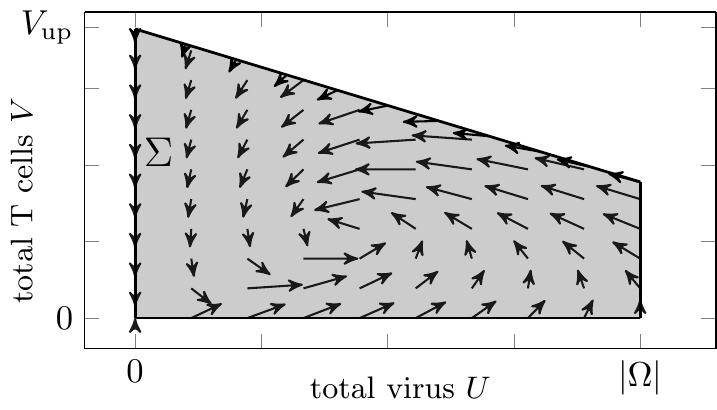}
\caption{Trapezoid $\Sigma$ in the phase space $(U,V)$ of the $L^1(\Omega)$-norms of a solution $(u,v)$ of system~(\ref{eq:sys}). The scaled vector field shows the dynamics of the reaction terms in Eq.~(\ref{eq:sys}). }
\label{fig:sigma}
\end{center}
\end{figure} 

The amount of T~cells $v$ is not only bounded in the sense of $L^1(\Omega)$ but also in the sense of $L^2(\Omega)$. 
This can be shown by regarding the time derivative of the functional 
\begin{align*}
\Psi(t)=\frac{1}{2}\|v(t, \mathbf{x})\|_{L^2(\Omega)}^2= \frac{1}{2} \int_\Omega v^2 \, \mathrm{d} \mathbf{x}.
\end{align*}

\begin{theorem}\label{thm:vL2}
Let $(u,v)$ be a solution of system~(\ref{eq:sys}). 
Then, the $L^2(\Omega)$-norm of $v$ is bounded for all $t >0$. 
\end{theorem}

\begin{proof}
The time derivative of the functional $\Psi$ is
\begin{align*}
\begin{aligned}
\Psi_{,t}&= \frac{\mathrm{d}}{\mathrm{d}t} \int_\Omega \frac{v^2}{2}  \, \mathrm{d} \mathbf{x} 
	= \int_\Omega v \cdot v_{,t} \, \mathrm{d} \mathbf{x} 
	=\int_\Omega v j[u] + v \cdot \eta (u-1) v + v \cdot  \beta \Delta v \, \mathrm{d} \mathbf{x} .
\end{aligned}
\end{align*}

Using Green's first identity and the zero-flux boundary conditions, we get
\begin{align*}
\beta \int_\Omega v  \Delta v \, \mathrm{d} \mathbf{x} 
	= \beta \int_{\partial \Omega} v \nabla v \cdot \mathbf{n} \, \mathrm{d} s - \beta \int_\Omega  \nabla v \cdot \nabla v \, \mathrm{d} \mathbf{x} 
	= - \beta \int_\Omega  \nabla v \cdot \nabla v \, \mathrm{d} \mathbf{x}  \leq 0 .
\end{align*}

Further, Remark~\ref{rem:zuj} provides an estimate for the integral of $j[u]$, which is
\begin{align*}
\Psi_{,t}
	&\leq  \int_\Omega v \cdot j [u] + \eta (u-1) v^2 \, \mathrm{d} \mathbf{x} \leq \delta \chi_\mathrm{max} U \int_\Omega v \, \mathrm{d} \mathbf{x} + \eta \int_\Omega (u-1) v^2 \, \mathrm{d} \mathbf{x} 	
\end{align*}
with $\displaystyle  \chi_\mathrm{max}  = \max_{\mathbf{x} \in \Omega} \chi_\Theta (\mathbf{x})$ according to Eq.~(\ref{eq:chi1}).
Now, Remark~\ref{rem:L1v} assures
\begin{align*}
\Psi_{,t} \leq \delta  \chi_\mathrm{max} \lvert \Omega \rvert V_\mathrm{up}+ \eta  \int_\Omega (u-1) v^2 \, \mathrm{d} \mathbf{x} = M - \eta \xi(t) \int_\Omega v^2 \, \mathrm{d} \mathbf{x}, 
\end{align*}
with the constant $M=\delta  \chi_\mathrm{max} \lvert \Omega \rvert V_\mathrm{up} $ and the weighted mean value $\xi(t)$ defined by 
\begin{align}\label{eq:meanvalue}
 \int_\Omega (1-u) v^2 \, \mathrm{d} \mathbf{x} =  \xi(t) \int_\Omega  v^2 \, \mathrm{d} \mathbf{x}.
\end{align}
The mean value $\xi(t)$ fulfills $0< 1- \theta \leq \xi(t) $ because of Lemma~\ref{lem:integral}. 
Finally, the functional $\Psi$ obeys the linear differential inequality
\begin{align}\label{eq:28}
\Psi_{,t}=\frac{\mathrm{d}}{\mathrm{d}t} \int_\Omega \frac{1}{2} v^2 \, \mathrm{d} \mathbf{x}  
	\leq M - \eta \xi(t) \int_\Omega  v^2 \, \mathrm{d} \mathbf{x}=M-2\eta \xi(t)  \Psi
\end{align}
with a positive decay rate $2 \eta \xi(t)$ which stays remote from 0.  
Eq.~(\ref{eq:28}) is a first order differential inequality, compare \cite{walter_ordinary_1998}, 
and $\Psi (t)$ is bounded by the solution of the linear first order differential equation $y^\prime = M- 2 \eta \xi(t) y$ with $\xi(t) \geq 1- \theta >0$.

Thus, the largest possible accumulation point of $\Psi$ is $\frac{M}{2 \eta (1- \theta)}$, and the functional $\Psi$ is bounded by $\frac{M}{ \eta (1- \theta)}$ after a transient phase.

Later in Sec.~\ref{sec:num}, we will use the estimate
\begin{align}\label{eq:estimate_psi2}
\Psi (t) \leq \Psi (0) \mathrm{e}^{-2 \eta \int_0^t   \xi (s) \, \mathrm{d} s} + \frac{M}{2 \eta (1- \theta)}  \leq \Psi(0) \mathrm{e}^{-2 \eta (1- \theta) t} 
\end{align}
for showing numerically the precision of the estimates. 
\end{proof}

Theorem~\ref{thm:sigma} and~\ref{thm:vL2} show, that the $L^1(\Omega)$- and the $L^2(\Omega)$-norms of $v$ are not only bounded for a time interval $[0,T)$ but for all time $t>0$. 
So in these norms, the solution is not blowing up. 

\subsection{$L^\infty(\Omega)$ bounds and global existence}\label{sec:linfty}
In this section, we show the boundedness of $v$ in $L^\infty(\Omega)$ for all $t>0$. 
With the boundedness of $v(t,\mathbf{x})$, the existence of a solution $(u,v)^\mathrm{T}$ with finite values is shown for all $t>0$.

We will prove, that there exists a value $v_\mathrm{max}$ with $v(t,x)\leq v_\mathrm{max}$ for all $ \mathbf{x} \in \Omega$ and all $ t \in [0, \infty)$.
For this purpose, a stationary problem is defined. 
Let $v^\star=v^\star(\mathbf{x})$ be a solution of 
\begin{align}\label{eq:probstat}
\begin{aligned}
- \beta \Delta v  &=\chi_\Theta (\mathbf{x}) - \frac{1}{\lvert \Omega \rvert}   &&   \text{for } \mathbf{x} \in \Omega, \\
 \nabla v \cdot \mathbf{n} & = 0 &&  \text{for } \mathbf{x } \in \partial \Omega .
\end{aligned}
\end{align}

System~(\ref{eq:probstat}) fulfills the solvability condition because the forces are equalized, see Eq.~(\ref{eq:chi1}) and 
\begin{align*}
\begin{aligned}
\int_\Omega \chi_\Theta (\mathbf{x}) - \frac{1}{\lvert \Omega \rvert}+ \beta \Delta v \, \mathrm{d} \mathbf{x}
	&= \int_\Omega \chi_\Theta (\mathbf{x}) \, \mathrm{d} \mathbf{x} - \int_\Omega  \frac{1}{\lvert \Omega \rvert}  \, \mathrm{d} \mathbf{x} + \beta \int_{\partial \Omega} \nabla v \cdot \mathbf{n} \, \mathrm{d} s\\
	&= 1- \frac{\lvert \Omega \rvert}{\lvert \Omega \rvert} + 0 = 0 .
\end{aligned}
 \end{align*}
 
 \begin{remark}
 Since the right-hand side $\chi_\Theta (\mathbf{x}) - \frac{1}{\lvert \Omega \rvert} $ in Eq.~(\ref{eq:probstat}) is a bounded piecewise continuous function and thus in $L^2(\Omega) \subset H^{-1}(\Omega)$, the existence of a weak solution $v^\star \in H^1(\Omega)$ is ensured, compare \cite{renardy_introduction_2010}. 
 \end{remark}
 
 \begin{remark}\label{rem:starconst}
 The solution $v^\star$ of Eq.~(\ref{eq:probstat}) has a free additive constant as always in pure Neumann problems. 
 In the following, we fix just one $v^\star$ with $\lVert v^\star (\mathbf{x}) \rVert_{L^1(\Omega)} =0$. 
 \end{remark}

  Now, we will show that the population $v= v(t,\mathbf{x})$ in Eq.~(\ref{eq:sys}) does not grow to infinity. 
 Even having already estimates for its $L^1(\Omega)$-norm, cf. Theorem~\ref{thm:sigma}, and for its $L^2(\Omega)$-norm, cf. Theorem~\ref{thm:vL2}, it is not trivial to give a pointwise bound. 
 Before we will do that in the later Theorem~\ref{thm:vinfty}, we collect some auxiliary results about solutions of partial differential equations with homogeneous Neumann boundary conditions. 
 
 \begin{lemma}\label{lemmaA}
 Let $v= v(t,\mathbf{x})$ be the solution of
 \begin{align}\label{eq:lemmaA}
 \begin{aligned}
 v_{,t} &= \beta \Delta v - a(t, \mathbf{x}) v + f(t, \mathbf{x}) &&   \text{for } \mathbf{x} \in \Omega , t>0 ,\\
 \nabla v \cdot \mathbf{n} & = 0 && \text{for } \mathbf{x} \in \partial \Omega, t >0, \\
 v(0, \mathbf{x}) & = 0 && \text{for } \mathbf{x} \in \Omega
 \end{aligned}
 \end{align}
 with $a(t, \mathbf{x}) \geq 0 $ and $\lvert f (t,\mathbf{x}) \rvert \leq C_f$ for all $\mathbf{x} \in \Omega, t >0$. 
 Then $\lvert \Delta v \rvert$ is bounded by a constant $C_c \in \mathbb{R}$ for all $\mathbf{x} \in \Omega$ and $t>0$. 
 \end{lemma}
 
 \begin{remark}
 Eq.~(\ref{eq:lemmaA}) is a heat conduction equation with the additional leveling term $-a(t , \mathbf{x}) v$, homogeneous Neumann boundary conditions, i.e. no heat flux over the boundary, and a  bounded heat source~$f$. 
 Hence, a physical point of view implies immediately a bounded curvature of $v$.
 \end{remark}

The mathematical argumentation starts with the Green's function $G=G(t,\mathbf{x}, \tau, \mathbf{y})$ of Eq.~(\ref{eq:lemmaA}), which is dominated by the singularity of the standard heat equation. 
Due to the Neumann boundary condition, there are no additional source terms at the boundary. 
The Laplacian $\Delta v(\mathbf{x})$ is the convolution of $\Delta_{\mathbf{x}} G$ with the bounded function $f$. 
This convolution can be estimated by a sum of spatial integrals over small domains and afterwards by time integration leading to terms in the Gauss' error function. 
Due to its technical effort, we omit the argumentation of the physically proven assertion of Lemma~\ref{lemmaA}. 
By the way, another possible argumentation uses a discretization of the Eq.~(\ref{eq:lemmaA}), where the eigenvalues and eigenvectors of the discretized differential operator $-(\Delta+a)$ can be estimated in a similarly technical argumentation. 
Then the limit case of a temporal step size tending to zero provides the assertion of Lemma~\ref{lemmaA} for every spatial discretization, and since $f$ is bounded also the limit situation of a vanishing grid size.

\begin{lemma}\label{lem:alt29}
Let $z: \Omega \rightarrow \mathbb{R}$ be a sufficiently smooth function with homogeneous Neumann boundary conditions.
The function $z$ fulfills $\lvert \Delta z (\mathbf{x}) \rvert \leq C_1$ for all $\mathbf{x} \in \Omega$ and $\lVert z \rVert_{L^1(\Omega)} \leq C_2$.
Then its values $z(\mathbf{x})$ are bounded by some $z_\mathrm{max} < \infty$ with $\lvert z (\mathbf{x}) \rvert \leq z_\mathrm{max}$ for all $\mathbf{x} \in \Omega$.
\end{lemma}

\begin{proof}
Such a function $z$ solves a boundary value problem 
\begin{align}
	\begin{aligned}
		- \Delta z (\mathbf{x}) & = \varrho (\mathbf{x}) && \text{for } \mathbf{x} \in \Omega, \\
		\nabla z  \cdot \mathbf{n} & = 0 && \text{for } \mathbf{x} \in \partial \Omega
	\end{aligned} 
\end{align}
with a source term with $\lvert \varrho (\mathbf{x}) \rvert \leq C_1$ for all $\mathbf{x} \in \Omega$.
We choose a Green's function $G(\mathbf{x}, \mathbf{y})$ with $G( \mathbf{x}, \mathbf{y}) \geq 0$ for all $\mathbf{x}, \mathbf{y} \in \Omega$.

With an additive constant $C_3$, we have
\begin{align*}
	z(\mathbf{x}) = \int_\Omega G( \mathbf{x}, \mathbf{y}) \varrho (\mathbf{y}) \, \mathrm{d} \mathbf{y} + C_3 \ \text{ and } \ \lvert z (\mathbf{x}) - C_3 \rvert \leq C_1 \int_\Omega G( \mathbf{x}, \mathbf{y})  \, \mathrm{d} \mathbf{y} = \tilde{\Phi} (\mathbf{x})
\end{align*}
with the smooth and bounded potential $\tilde{\Phi}$ for a constant source term.

There is a value $C_{3,\mathrm{max}} < \infty$ so that the condition $\lVert z \rVert_{L^1(\Omega)} \leq C_2$ is not fulfilled for any $C_3\geq C_{3,\mathrm{max}}$.
Consequently, we get 
\begin{align*}
\lvert z(\mathbf{x}) \rvert \leq C_{3,\mathrm{max}} +\max_{\mathbf{x} \in \Omega} \tilde{\Phi} (\mathbf{x}) = z_\mathrm{max} < \infty.
\end{align*}
\end{proof}

This lemma shows the boundedness of a stationary problem which displays in parts the inflow of T~cells in a certain region. 
In the next step, we use this result for showing the existence of an upper bound for $v$ in a time-dependent setting which still abstracts from the coupled reaction diffusion system in (\ref{eq:sys}). 

 \begin{lemma}\label{lem:22}
 Let $v= v(t, \mathbf{x})$ be the solution of
 \begin{align}\label{eq:lemmaC}
 	\begin{aligned}
		v_{,t} &= \beta \Delta v - a ( t, \mathbf{x}) v + f(t, \mathbf{x}) && \text{for } \mathbf{x} \in \Omega, t >0,\\
		\nabla v \cdot \mathbf{n} & = 0 &&  \text{for } \mathbf{x} \in \partial \Omega, t >0,\\
		v(0, \mathbf{x}) & = v_0(\mathbf{x}) && \text{for } \mathbf{x} \in \Omega
	\end{aligned}
\end{align} 
with $a(t,\mathbf{x}) \geq 0$ and $\lvert f(t, \mathbf{x}) \rvert \leq C_f$ for all $\mathbf{x} \in \Omega , t >0$ and bounded initial conditions $v_0(\mathbf{x})$. 
Furthermore, it shall be known that $\lVert v \rVert_{L^1(\Omega)} \leq C_2$ is bounded.

Then, there is a bounded $v_\mathrm{max}$ with $\lvert v (t, \mathbf{x}) \rvert \leq v_\mathrm{max}$ for all $\mathbf{x} \in \Omega$ and all $t>0$.
 \end{lemma}
 
 At every maximum point of $v(t, \cdot)$, we find $v_{,t} \leq f \leq C_f$, and the maximum $\displaystyle \max_{\mathbf{x} \in \Omega} v(t, \cdot)$ grows at most linearly.
 So the following proof excludes an infinite growth of $f$ for $t \to \infty$.
 
 \begin{proof}
 System~(\ref{eq:lemmaC}) is a linear differential equation and the solution $v$ decomposes into $v= v_\mathrm{hom} + v_\mathrm{part}$.
  The function $v_\mathrm{hom}$ obeys the homogeneous equation with $f \equiv 0$ and fulfills the initial conditions $v_0$.
  The function $v_\mathrm{part}$ solves the system~(\ref{eq:lemmaA}) from Lemma~\ref{lemmaA}.
  
  The function $v_\mathrm{hom}$ follows the maximum principle
  \begin{align*}
  	\max_{\mathbf{x} \in \Omega, t \geq 0} \lvert v_\mathrm{hom}(t, \mathbf{x}) \rvert = \max_{\mathbf{x} \in \Omega} \lvert v_0(\mathbf{x}) \rvert
\end{align*}
and stays bounded. 

Lemma~\ref{lemmaA} says that $v_\mathrm{part} (t, \cdot)$ has a bounded Laplacian $\lvert \Delta v_\mathrm{part} (t, \mathbf{x}) \vert \leq C_1$ for all $\mathbf{x} \in \Omega$.
 Lemma~\ref{lem:alt29} assures that $v_\mathrm{part}$ is bounded for all times by a $z_\mathrm{max} \in \mathbb{R}$.
 Together with the boundedness of $v_\mathrm{hom}$, we find
 \begin{align*}
 	\lvert v(t, \mathbf{x}) \rvert \leq \max_{\mathbf{x} \in \Omega} \lvert v_0(\mathbf{x}) \rvert +z_\mathrm{max} .
 \end{align*}
 \end{proof}
 
 Following, we adapt this result for the coupled reaction diffusion system modeling the dynamics of a liver infection. 
 
 \begin{theorem}\label{thm:vinfty}
 The solution $v$ of Eq.~(\ref{eq:sys}) is bounded by a finite value $v_\mathrm{max}$.
 \end{theorem}

\begin{proof}
We decompose $v= v(t, \mathbf{x})$ into 
 \begin{align*}
v(t,\mathbf{x}) = \delta v^\star(\mathbf{x}) U(t) +\tilde{v}(t,\mathbf{x})
\end{align*}
and the evolution equation for $v$ in Eq.~(\ref{eq:sys}) transforms into
\begin{align}\label{eq:dtzerleg}
\begin{aligned}
v_{,t} &= \delta U^\prime(t) v^\star(\mathbf{x}) + \tilde{v}_{,t} = j[u] - \eta (1-u) \left (\delta U v^\star + \tilde{v} \right ) + \delta U \beta \Delta v^\star + \beta \Delta \tilde{v}.
\end{aligned}
\end{align}
Due to Eq.~(\ref{eq:inflow}) and the stationary solution $v^\star$ of Eq.~(\ref{eq:probstat}), Eq.~(\ref{eq:dtzerleg}) simplifies to
\begin{align}\label{eq:dtsimp}
\tilde{v}_{,t} = \beta \Delta \tilde{v} - \eta (1-u) \tilde{v} + f , 
\end{align}
where the exogenous influence
\begin{align}\label{eq:fexo}
f= \frac{\delta U}{\lvert \Omega \rvert} - \eta (1-u) \delta U v^\star - \delta U^\prime v^\star
\end{align}
for the standard diffusion problem in Eq.~(\ref{eq:dtsimp}) is a function in $t$ and $\mathbf{x}$ for fixed $u$ and thus fixed $U$ and $U^\prime$ are as well as $v^\star$. 
Eq.~(\ref{eq:dtsimp}) is completed with homogeneous Neumann boundary conditions for $\tilde{v}$.

That means, we regard $u$ to be given and investigate Eq.~(\ref{eq:dtsimp}) as a diffusion problem for $\tilde{v}= \tilde{v}(t, \mathbf{x})$ with the exogenous influence $f= f(t, \mathbf{x})$.
Since
\begin{align*}
U^\prime (t) = \int_\Omega u w(u) \, \mathrm{d} \mathbf{x} - \gamma \int_\Omega uv \, \mathrm{d} \mathbf{x}
\end{align*}
and $\lvert w(u) \rvert \leq \nu $ with $\nu \in \mathbb{R}_+$, cf. Remark~\ref{rem:zuw}, we see
\begin{align*}
\lvert U^\prime (t) \rvert \leq \nu U(t) + \gamma V(t)
\end{align*}
for all $t$. 
Together with the boundedness of $U$ and $V$ for all $t$, the exogenous influence $f$ in Eq.~(\ref{eq:fexo}) is bounded by some constant $C_\mathrm{f} \geq \lvert f(t, \mathbf{x}) \rvert$ for all $t > 0 $ and $\mathbf{x} \in \Omega$. 

Now with Eq.~(\ref{eq:dtsimp}), we find 
\begin{align}\label{eq:l1tilde}
V_\mathrm{up} \geq \int_\Omega v(t,\mathbf{x}) \, \mathrm{d} \mathbf{x} = \delta U(t) \int_\Omega v^\star (\mathbf{x}) \, \mathrm{d} \mathbf{x} + \int_\Omega \tilde{v}(t, \mathbf{x}) \, \mathrm{d} \mathbf{x} = \int_\Omega \tilde{v} (t, \mathbf{x}) \, \mathrm{d} \mathbf{x}.
\end{align}

Now, Eq.~(\ref{eq:dtsimp}) fulfills all conditions of Lemma~\ref{lem:22}, namely $a(t,\mathbf{x}) = \eta (1-u) \geq 0$ and $\lvert f(t ,\mathbf{x}) \rvert \leq C_f$ for a heat equation with homogeneous Neumann boundary conditions. 

Consequently, there exists a maximal value $\tilde{v}_\mathrm{max} \geq \tilde{v} (t, \mathbf{x})$ for all $t>0$ and all $\mathbf{x} \in \Omega$, and we get 
\begin{align*}
v(t,\mathbf{x}) \leq \delta \lvert \Omega \rvert \max v^\star(\mathbf{x}) + \tilde{v}_\mathrm{max} = v_\mathrm{max} \in \mathbb{R}.
\end{align*}
 
\end{proof}

\begin{remark}\label{rem:vlinfty}
The solution $v$ of Eq.~(\ref{eq:sys}) is bounded by a finite value $v_\mathrm{max}$. 
Finally, the solution $v(\cdot, \mathbf{x})$ is a $L^\infty(\Omega)$-function. 
\end{remark}

 In this section, we have proven the boundedness of the solution of Eq.~(\ref{eq:sys}). 
 While the boundedness of $u$ was a result of the used growth function and therefore allows interpretation as a concentration, the boundedness of $v$ was not obvious. 
Using the oppositely acting mechanisms in the reaction functions and the boundedness of $u$, we first showed the boundedness of $v$ in $L^1(\Omega)$. 

We provided a bounded estimate for the $L^2(\Omega)$-norm of $v$ by using the mean-value theorem of integration and the boundedness of $\lVert v (\cdot, \mathbf{x}) \rVert_{L^1(\Omega)}$. 

For proving the boundedness of $v$ in $L^\infty(\Omega)$, we separated $v(t, \mathbf{x})= \delta v^\star (\mathbf{x}) U(t) + \tilde{v}(t, \mathbf{x})$ into different functions. 
One component, $v^\star$, of the functions was the solution of a stationary problem covering the space-dependent function modeling the inflow area of the liver structure. 
By showing the boundedness of all components of $v$, we proved in Theorem~\ref{thm:vinfty} that $v$ has a finite maximal value.

Applied to the modeling of liver infections, the result of Rem.~\ref{rem:vlinfty} says that the amount of T~cells is bounded by a finite value. 
The T~cells attack the virus by triggering the programmed cell death. 
This leads to inflammation in the liver tissue and can cause secondary diseases like cancer. 
Besides, a too high amount of T~cells might cause a sepsis. 
Rem.~\ref{rem:vlinfty} gives the fact, that the immune reaction is bounded but does not give a finite value. 
Therefore, Rem.~\ref{rem:vlinfty} justifies the use of model in the sense, that the immune reaction remains bounded for all time.
It is a first step towards quantitative and finer estimates for $v_\mathrm{max}$ which contain information about the occurrence of sepsis. 

Therefore, the next section provides numerical evaluations on the $L^1(\Omega)$ and $L^2(\Omega)$ estimates. 

\section{Numerical evaluation of the estimates}\label{sec:num}

Oftentimes, estimates used in analytical results are rather rough. 
In this section, we show numerical simulations of the estimates and the exact value. 

First, we evaluate the estimation of the domain $\Sigma$ as maximal $L^1(\Omega)$-norms. 
In Fig.~\ref{fig:esti_sigma}, the trajectories of the two solutions from Fig.~\ref{fig:healing} and~\ref{fig:chronic} provide the total amount $U$ and $V$.
They are compared to the estimated $\Sigma$ following Theorem~\ref{thm:sigma}. 

\begin{figure}[bhtp]
\begin{center}
\includegraphics[width=\textwidth]{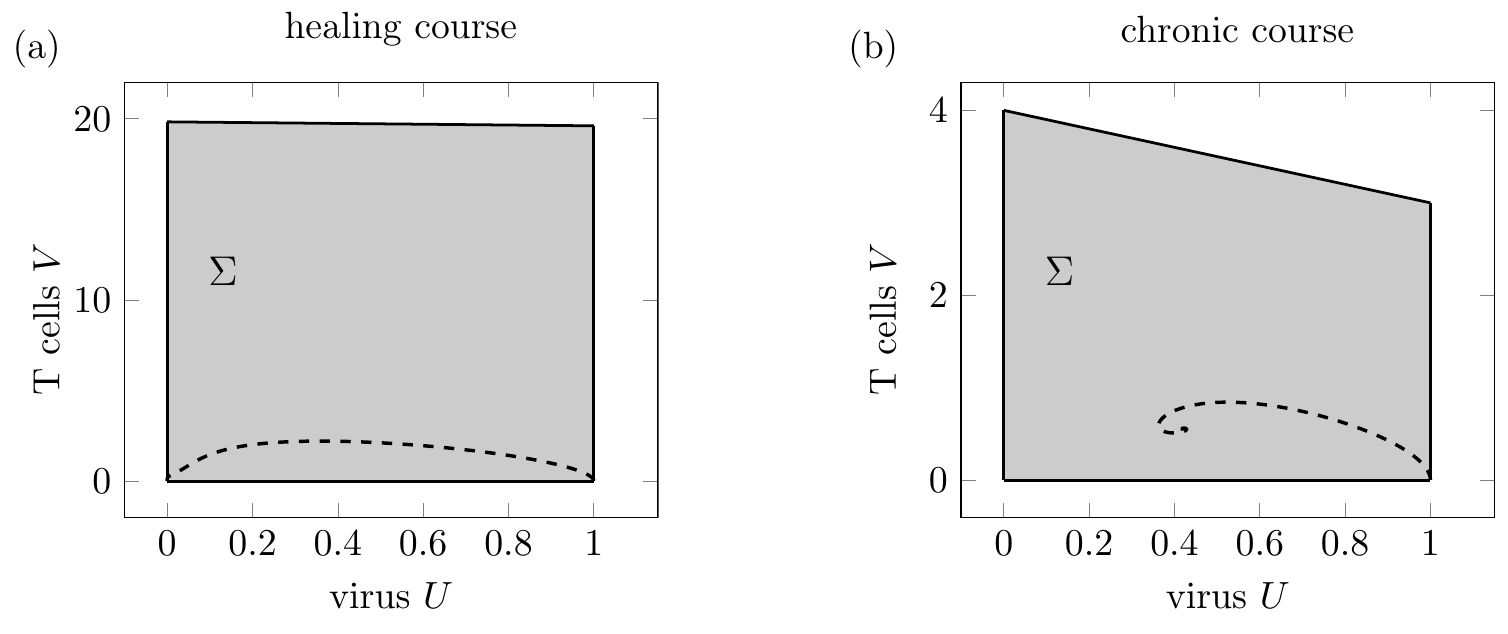}
\caption{Comparison of the trajectories of different solutions in phase space $(U, V)$ and the trapezoid~$\Sigma$.
(a) Healing course, see Fig.~\ref{fig:healing} for the parameters. The upper value $V_\mathrm{up}$ is given by $V_\mathrm{up}= 19.833$.
(b)~Chronic course, see Fig.~\ref{fig:chronic} for the parameters. The upper value $V_\mathrm{up}$ is given by $V_\mathrm{up}= 4$.}
\label{fig:esti_sigma}
\end{center}
\end{figure}

Fig.~\ref{fig:esti_sigma} shows, that the upper bound of $\eta U + \gamma V \leq \gamma V_\mathrm{up}$ is a rather rough estimate for the $L^1(\Omega)$-norms of the solutions. 
In the numerical simulations in Fig.~\ref{fig:healing} and Fig.~\ref{fig:chronic}, which are as well used in Fig.~\ref{fig:esti_sigma}, the initial conditions are $u(0,\mathbf{x})\equiv 1$ and $v(0,\mathbf{x}) \equiv 0$.
A solution with larger initial conditions $V(0)$  would reach closer to the upper bound given by $\Phi=  \gamma V_\mathrm{up}$. 
As shown in Theorem~\ref{thm:sigma}, the $L^1(\Omega)$-norms of every solution with $(U(0), V(0)) \in \Sigma$ stay in $\Sigma$. 

Theorem~\ref{thm:vL2} gives in Eq.~(\ref{eq:estimate_psi2}) an estimate for the $L^2(\Omega)$-norm of $v$. 
By using  $\xi (t) $ as solution of Eq.~(\ref{eq:meanvalue}) and directly Eq.~(\ref{eq:28}) we get the approximation
\begin{align*}
\Psi(t) = \frac{1}{2} \lVert v(t,\mathbf{x}) \rVert_{L^2(\Omega)}^2 \leq  \mathrm{e}^{-2 \eta \int_0^t   \xi (s) \, \mathrm{d} s} \left ( M \int_0^t \mathrm{e}^{2 \eta \int_0^\tau \xi (s) \, \mathrm{d} s} \, \mathrm{d} \tau + \Psi(0) \right) 
\end{align*}
with a parameter-dependent constant $M= \delta \lvert \Omega \rvert \chi_\mathrm{max} V_\mathrm{up}$.
The inequality $0 < \xi(t) \leq 1$ yields according to {Lemma~\ref{lem:integral}}. %

In the cases of the two regarded simulations in Fig.~\ref{fig:healing} and Fig.~\ref{fig:chronic} of Eq.~(\ref{eq:sys}), the initial value $\Psi(0)$ is zero, because $v(0,\mathbf{x})=0$ for all $\mathbf{x} \in \Omega$. 
Therefore, we compare the $\Psi(t)= \frac{1}{2} \lVert v(t,\mathbf{x}) \rVert_{L^2(\Omega)}^2$ with the functional 
\begin{align*}
E(\xi(t)) = \mathrm{e}^{-2 \eta \int_0^t   \xi (s) \, \mathrm{d} s}  \delta \lvert \Omega \rvert V_\mathrm{up} \int_0^t \mathrm{e}^{2 \eta \int_0^\tau \xi (s) \, \mathrm{d} s} \, \mathrm{d} \tau .
\end{align*}

\begin{figure}[thbp]
\begin{center}
\includegraphics[width=\textwidth]{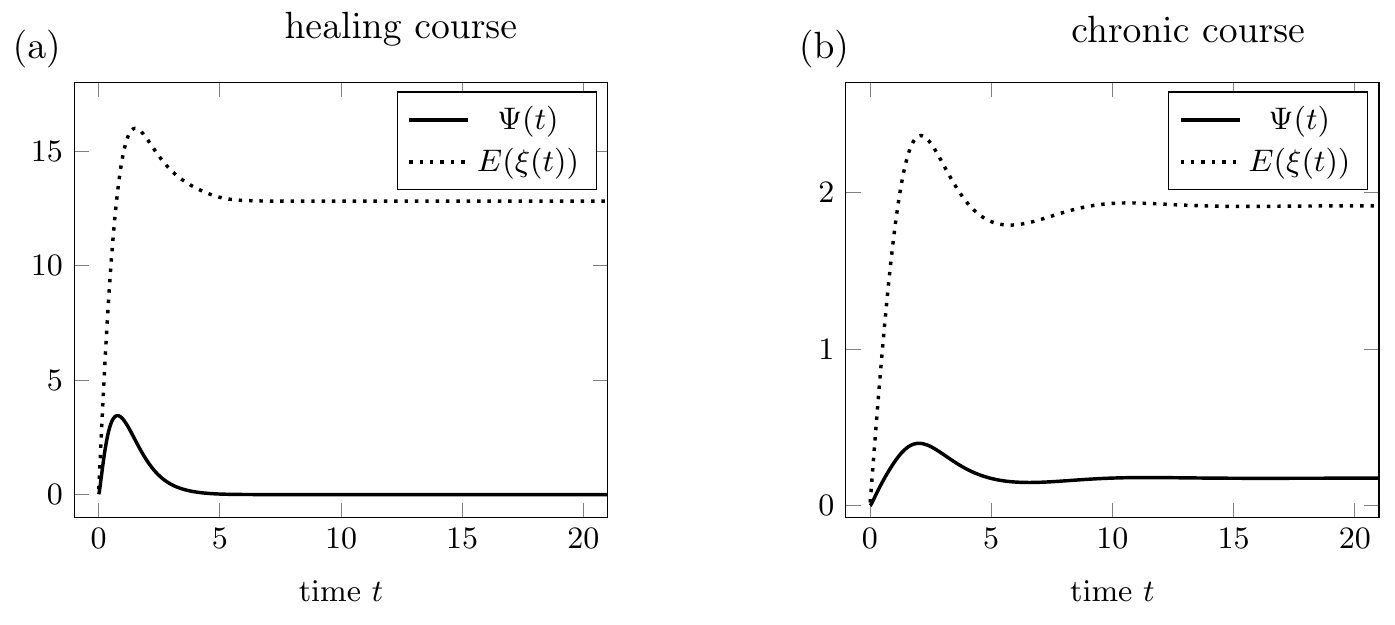}
\caption{Comparison $\Psi(t)$ and $E(\xi(t))$. 
(a) Healing course, see Fig.~\ref{fig:healing} for the parameters.
(b) Chronic course, see Fig.~\ref{fig:chronic} for the parameters.
In both cases, $E(\xi(t))$ overestimates the functional $\Psi(t)$.}
\label{fig:esti_l2}
\end{center}
\end{figure}

In both cases in Fig.~\ref{fig:esti_l2}, the estimation $E$ is rather large compared to the functional $\Psi(t)$.
Nevertheless, the estimate is a good approximation of scale of the maximal value of $\Psi$.

The numerical simulations show that the used estimates are rather loose even if they were sufficient for gaining the analytical existence results. 
The estimations exceed the values of the numerical solutions. 
In context of liver infections, the estimations can be regarded as worst case scenario. 
A medical treatment in the light of the worst case scenario might lead to a longer lasting infection course but decrease the risk of a sepsis.

\section{Conclusions}

With the aim to modeling the dynamics of liver infections as an interplay between virus and T~cells, a reaction diffusion system was presented in \cite{kerl_reaction_2012}.
A non-local term in the reaction function describes the inflow of T~cells depending on the total virus amount in the domain. 
The model abstracts from the cell scale with many unknown mechanisms to a mesoscopic length scale. 
On this scale, the mathematical description contains a space-dependent term which leads to a new problem concerning the analysis of reaction diffusion equations. 
Additionally, the reaction terms contain oppositely acting mechanisms resulting in a feedback loop for the increase of T cells.

The aim of this paper was to prove the existence of bounded solutions for all time and to provide a proof with interpretable intermediate steps. 
Therefore, we started with a local existence theorem and some properties of a weak solution. 
Then, we showed the boundedness of the solution in the $L^1(\Omega)$-norm and in the $L^2(\Omega)$-norm. 
Both results are based on the interplay of the two species in the population dynamics model and the oppositely acting mechanisms of growth and decay. 
We defined a stationary problem for showing the boundedness of the solution in the $L^\infty(\Omega)$-norm. 

In Sec.~\ref{sec:num}, we evaluated the sharpness of the used estimates in the proofs. 
The numerical simulations show that the estimates are rather loose for the regarded cases. 

In the light of the application, modeling liver infections, the estimation can be seen as a worst case scenario. 
The boundedness of the $L^\infty(\Omega)$-norm of $v$ is a feature uprating the model for liver infections. 

A further investigation could be the improving of the used estimated such that the difference between estimate and real value of the functionals becomes smaller. 
Another possible extension is the application of the estimates for a wider class of integro-partial differential equations. 

\subsection*{Financial disclosure}

None reported.

\subsection*{Conflict of interest}

The authors declare no potential conflict of interests.

\section*{References}

\bibliography{paper_existence_arXiv}

\end{document}